\documentclass[a4paper,11pt,reqno]{amsart}%,oneside

\usepackage[ansinew]{inputenc}
\usepackage[english]{babel}
\usepackage{amsmath}
\usepackage{amssymb}
\usepackage{amsthm}
\usepackage{mathrsfs}
\usepackage[all]{xy}
\usepackage{verbatim}
\usepackage{arydshln}
\usepackage{multicol}
\usepackage{enumitem}

\usepackage[bookmarks=true,bookmarksopen=false]{hyperref}
\hypersetup{
   pdftitle = {},
   pdfauthor = {Matthew Gelvin},
   pdfpagemode = {UseOutlines},
   pdfstartview = {FitH},
   pdfborder = {0 0 0},
   backref = {true},
   colorlinks = {true},
   urlcolor = {blue},
   citecolor = {purple},
   linkcolor = {blue},
   pdftoolbar = {true},
}

\usepackage{mathtools}
\usepackage{ifthen}
\usepackage{tikz}
\usepackage{todonotes}
\usepackage{bookmark}

\allowdisplaybreaks

\theoremstyle{plain}
\newtheorem{theorem}{Theorem}%[section]
\newtheorem{prop}[theorem]{Proposition}
\newtheorem{lemma}[theorem]{Lemma}
\newtheorem{cor}[theorem]{Corollary}

\theoremstyle{definition}
\newtheorem{definition}[theorem]{Definition}

\theoremstyle{remark}

\DeclareMathOperator{\Stab}{Stab}%\newcommand\Stab{Stab}

\DeclareMathOperator{\Hom}{Hom}

\DeclareMathOperator{\Syl}{Syl}

\newcommand{\id}{\mathrm{id}}

\def\cF{\mathcal F}

\def\NN{\mathbb N}

\def\id{\mathrm{id}}

\usetikzlibrary{arrows,matrix,decorations,positioning,calc}
\tikzset{dot/.style={circle,fill=black,thick,inner sep=0pt,minimum size=1mm,draw}}
\tikzset{arrow/.style={semithick,>=stealth',shorten >=1pt,shorten <=1pt}}
\tikzset{equal/.style={arrow,double distance=2pt}}

\title{An observation on the module structure of block algebras}
\author{Matthew Gelvin}
\date{\today}

\begin{document}

\maketitle

\noindent{\bf Abstract.}  Let $B$ be a $p$-block of the finite group $G$.  We observe that the $p$-fusion of $G$ constrains the module structure of $B$:  Any basis of $B$ that is invariant under the left and right multiplications of a chosen Sylow $p$-subgroup $S$ of $G$ must in fact form a semicharacteristic biset for the fusion system on $S$ induced by $G$.  The parameterization of such semicharacteristic bisets can then be applied to relate the module structure and defect theory of $B$.\\

\noindent{\bf\S 0. Introduction.} 

\noindent Let $G$ be a finite group and $S$ a Sylow $p$-subgroup of $G$.  The left and right multiplicative actions of $S$ on $G$ give a partition of $G$ by double cosets $G=\coprod\limits_{i=1}^mSg_iS$, for some chosen set of representatives $\{g_i\}$.  Each double coset is a transitive $(S,S)$-biset, and thus this partition is the orbit decomposition of the $(S,S)$-biset $_SG_S$.

If $k$ is an algebraically closed field of characteristic $p$, the group algebra $kG$ decomposes as a direct sum of its block algebras:  $kG=B_0\oplus B_1\oplus\ldots\oplus B_n$.  Standard results, summarized in Proposition \ref{prop:key_facts}, imply that each  $B_j$ possesses a $k$-basis $X_j$ that is stable under left and right $S$-multiplication.  Such an \emph{$S$-invariant} $k$-basis is itself an $(S,S)$-biset, and the disjoint union $\coprod\limits_{j=1}^n X_j$ yields an $(S,S)$-biset abstractly isomorphic to the $S$-invariant $k$-basis $_SG_S$ of $kG$.  Thus the $\{X_j\}$ can be viewed as an $(S,S)$-partition of $_SG_S$ and we can group the $\{Sg_iS\}$ so as to recover the $(S,S)$-orbit decomposition of each $X_i$, with the proviso that none of this is canonical.  

Let us from now on focus on a particular block algebra $B$ with $S$-invariant $k$-basis $X$.  The purpose of this note is to show that the $(S,S)$-biset structure of $X$ is not arbitrary, in that it is to some extent controlled by $p$-fusion in $G$.  More precisely:

\begin{theorem}\label{thm:main}
If $X$ is an $S$-invariant  $k$-basis of the block algebra $B$ and $\cF=\cF_S(G)$ is the $G$-fusion system on $S$, then $X$ is an $\cF$-semicharacteristic $(S,S)$-biset.
\end{theorem}

%For $\cF$ an arbitrary saturated fusion system, the monoid of isomorphism classes of $\cF$-semicharacteristic bisets was shown to have a natural basis in \cite{GelvinReehMinimalCharacteristicBisets}.  Thus Theorem \ref{thm:main} imposes new structural constraints on the module structure of block algebras.  We argue for the importance of the main observation by using these constraints to give new proofs of a few known results in block theory.

Before we begin to prove Theorem \ref{thm:main}, we verify that our claim regarding the partition of $_SG_S$ in terms of bases of block algebras holds, which amounts to proving the existence and uniqueness, as an $(S,S)$-biset, of an $S$-invariant $k$-basis for $B$ (Proposition \ref{prop:block_invariant_basis}).  This follows from the basic theory of $p$-permutation modules, summarized in \S1.

In \S2 we give the definition of $\cF$-semicharacteristic biset and related notions.

In \S3 we prove Theorem \ref{thm:main}.  In doing so, we make use of the relationship between $G$ and the algebra  structure of $B$.  If $b\in Z(kG)$ is the block idempotent corresponding to  $B$, every  $g\in G$ commutes with $b$, and so $(g\cdot b)(g^{-1}\cdot b)=b$.  As $b$ is the identity element of  $B$,  the assignment $g\mapsto g\cdot b$ yields a group map $G\to B^\times$.  This makes $B$ an \emph{interior $G$-algebra}.  

We use this fact repeatedly and without further comment beyond a point on notation:   Multiplication in our algebras is indicated by concatenation of elements, while the symbol ``$\ \cdot\ $''  is reserved for  the action of an element of a group on an element of an algebra.  (We will occasionally use ``$\ \odot\ $'' for the same, when multiple  group actions must be compared.)  For $g\in G$, we  write $\overline g$ for the image of $g$ in the unit group of an interior $G$-algebra $A$, so that by definition $g\cdot a=\overline ga$ for all $a\in A$.

Finally, in \S4 we combine Theorem \ref{thm:main} with the parameterization of $\cF$-semicharacteristic bisets from \cite{GelvinReehMinimalCharacteristicBisets}  to impose structural constraints on the underlying $(kS,kS)$-bimodule of $B$.  We also use the main observation to give new perspectives on a few basic results in the theory of defect groups of blocks.

Thanks are due to Laurence Barker and Justin Lynd, whose independent collaborations with the author suggested the main result of this note as a observation of potential interest.\\

\noindent{\bf \S1.  $p$-permutation modules.}

\noindent Let $H$ be a finite group and $M$ a finite dimensional $kH$-module. $M$ is a \emph{$p$-permutation $H$-module} if for any $p$-subgroup $P\leq H$, $M$ possesses a $k$-basis $Y=Y_P$ that is invariant under the action of $P$.  Thus $Y$ is a finite $P$-set and as such can be written 
$$Y=\coprod\limits_{[Q]\leq P} c_Q(Y)\cdot [P/Q],$$
 where the coproduct is index by the $P$-conjugacy classes of subgroups of $P$, $c_Q(Y)\in\NN$, and $[P/Q]$ denotes the transitive $P$-set having a point with stabilizer $Q$. As $kP$-modules, 
 $$M\cong\bigoplus\limits_{[Q]\leq P}c_Q(Y)\cdot k[P/Q].$$
  A result of Green \cite[Lemma 2.3a]{GreenBlocksOfModularRepresentations} states that each $k[P/Q]$ is indecomposable as a $kP$-module, and that moreover if $k[P/Q]\cong k[P/R]$ as $kP$-modules then $P/Q\cong P/R$ as $P$-sets.  The Krull-Schmidt Theorem then implies:
  
\begin{prop}\label{prop:key_facts}
Let $H$ be a finite group, $M$ a $p$-permutation $kH$-module, and $P\leq H$ a $p$-subgroup.
\begin{enumerate}
\item\label{proppart:invariant_bases_well_defined} If $Y$ and $Y'$ are $P$-invariant $k$-bases of $M$, then $Y\cong Y'$ as $P$-sets.
\item\label{proppart:summands} If $N$ is a direct summand of $M$, then $N$ is a $p$-permutation $kH$-module.
\item\label{proppart:subsets} If $N$ is a direct summand of $M$ and $Z\subset N$, $Y\subset M$ are $P$-invariant $k$-bases, then $Z$ is isomorphic to a $P$-subset of $Y$.
\end{enumerate} 
\end{prop}

We  apply these results to the $(G\times G)$-module $kG$, with action  $(g_1,g_2)\cdot x=g_1\cdot x\cdot g_2^{-1}$.  The natural $k$-basis $G$ of $kG$ is clearly invariant under this action, hence under the restricted action to any $p$-subgroup of $G\times G$.  Thus $kG$ is a $p$-permutation $(G\times G)$-module.

If $b$ is a block idempotent of $kG$, we have $kG\cong kGb\oplus kG(1-b)$ as $kG$-modules.  In particular, the corresponding block algebra $B=kGb$ is a direct summand of $kG$, so $B$ is a $p$-permutation $G$-module by Proposition \ref{prop:key_facts}\ref{proppart:summands}.  In particular, for $S\in\Syl_p(G)$, there is an $(S\times S)$-invariant $k$-basis $X$ of $B$.  The $(S\times S)$-action on $X$ is equivalent to endowing $X$ with the structure of an $(S,S)$-biset. We will freely move between these notions without comment.  When viewed as an $(S,S)$-biset, $X$ is our $S$-invariant $k$-basis of $B$.

% In addition, we shall say a $k$-basis of $B$ is \emph{$S$-invariant} if it is invariant under the above $(S\times S)$-action.

We summarize the implications Proposition \ref{prop:key_facts}\ref{proppart:invariant_bases_well_defined}-\ref{proppart:subsets}:

\begin{prop}\label{prop:block_invariant_basis}
If $B$ is a block algebra of $kG$ and $S\in\Syl_p(G)$, then $B$ possesses an $S$-invariant $k$-basis $X$.  Such an $S$-invariant $k$-basis is uniquely determined up to isomorphism of $(S,S)$-bisets.  Moreover, $X$ is isomorphic to a $(S,S)$-subbiset of ${_SG_S}$.
\end{prop}
\newpage
\noindent{\bf \S2.  $\cF$-semicharacteristic bisets}. 

\noindent Let $\cF=\cF_S(G)$ be the fusion system on $S$ induced by $G$: $\cF$ is the category whose objects are the subgroups of $S$ and whose homsets are given by 
\[
\cF(P,Q)=\Hom_G(P,Q)=\left\{\varphi:P\to Q\ \big|\ \exists\ g\in G\textrm{ such that }\varphi=c_g|_P\right\},
\]
where $c_g:G\to G:x\mapsto gxg^{-1}$ is (left) conjugation by $g$.  An $\cF$-characteristic biset is an abstraction of the natural $(S,S)$-biset ${_SG_S}$ that controls the structure of $\cF$. 

Some basic terminology is needed to give the full definition.  Let $\Omega$ be a finite $(S,S)$-biset.

The \emph{opposite biset} of $\Omega$ is the $(S,S)$-biset $\Omega^\circ$ whose underlying set is $\Omega$ and whose left and right $S$-actions are given by $s_1\odot\omega\odot s_2=s_2^{-1}\cdot\omega\cdot s_1^{-1}$.  We say $\Omega$ is \emph{symmetric} if $\Omega\cong\Omega^\circ$ as $(S,S)$-bisets.

The \emph{point-stabilizer} of $\omega\in\Omega$ is the $S\times S$-stabilizer of $\omega$.  If $\Omega$ is thought of as an $(S,S)$-biset, this is the subgroup of $S\times S$ defined by $\Stab(\omega):=\left\{(s_1,s_2)\in S\times S\ \big|\ s_1\cdot\omega=\omega\cdot s_2\right\}$.  If $P\leq S$ and $\varphi:P\hookrightarrow S$ is a group monomorphism, the \emph{twisted diagonal subgroup} defined by $P$ and $\varphi$  is  $(\varphi,P):=\left\{(\varphi(u),u)\ \big|\ u\in P\right\}\leq S\times S$.  
 
 $\Omega$ is \emph{bifree} if the left and right $S$-actions on $\Omega$ are individually free.
 If $\Omega$ is bifree, then every $\omega\in\Omega$ has a twisted diagonal subgroup as its point-stabilizer; we shall write $(c_\Omega,S_\Omega)$ for $\Stab(\omega)$ in this case.  This notation comes from the example $\Omega={_SG_S}$, where an element $g\in {_SG_G}$ has point-stabilizer $(c_g,S_g)$, for $S_g:=S\cap{S^g}$ the largest subgroup of $S$ left-conjugated into $S$ by $g$.
 
 If $P\leq S$, let ${_P\Omega_S}$ be the $(P,S)$ biset whose left $P$-action comes from restriction of the  left $S$-action.   If  $\varphi:P\hookrightarrow S$ is a group monomorphism, ${_P^\varphi\Omega_S}$ is the $(P,S)$-biset whose left $P$-action is realized by first twisting by  $\varphi$:  For all $u\in P$, $s\in S$, and $\omega\in\Omega$, set $u\odot\omega\odot s:=\varphi(u)\cdot\omega\cdot s$. The $(S,P)$-bisets $_S\Omega_P$ and ${_S\Omega_P^\varphi}$ are defined similarly.

We can now give the precise definition of $\cF$-(semi)characteristic bisets.  This notion is due to Linckelmann and Webb, who formulated it in terms of abstract fusion systems.  As we  deal only with fusion system realized by finite groups, we shall make no further commentary on the more general situation.  See, e.g., \cite{AKO} for the complete picture.

\begin{definition}
Let $\cF$ be a saturated fusion system on the $p$-group $S$ and let $\Omega$ be an $(S,S)$-biset.
\begin{itemize}
\item $\Omega$ is \emph{$\cF$-generated} if for all $\omega\in\Omega$ with point-stabilizer $(c_\omega,S_\omega)$, the group map $c_\omega:S_\omega\to S$ satisfies $c_\omega\in\cF(S_\omega,S)$.
\item $\Omega$ is \emph{$\cF$-invariant} if for all $P\leq S$ and $\varphi\in\cF(P,S)$, ${_P^\varphi \Omega_S}\cong{_P\Omega_S}$ as $(P,S)$-bisets and ${_S\Omega_P^\varphi}\cong{_S\Omega_P}$ as $(S,P)$-bisets.
\item $\Omega$ is \emph{$\cF$-semicharacteristic} if
\begin{enumerate}
\item\label{cond:i} $\Omega$ is bifree,
\item\label{cond:ii} $\Omega$ is symmetric,
\item\label{cond:iii} $\Omega$ is $\cF$-generated, and
\item\label{cond:iv} $\Omega$ is $\cF$-invariant.
\end{enumerate}
\item $\Omega$ is \emph{$\cF$-characteristic} if $\Omega$ is $\cF$-semicharacteristic and in addition satisfies
\begin{enumerate}[resume]
\item\label{cond:v} $|\Omega|/|S|$ is prime to $p$.  
\end{enumerate}
\end{itemize}
\end{definition}

The existence of $\cF$-characteristic bisets for saturated fusion systems was shown in, e.g., \cite[Proposition 5.5]{BLO2}.  Moreover, the converse that if an $\cF$-characteristic biset exists then $\cF$ must be saturated was proved in \cite[Proposition 21.9]{PuigBook}, albeit with different terminology.

Observe that the ur-example $_SG_S$ is indeed an $\cF=\cF_S(G)$-characteristic biset:
\begin{enumerate}
\item $_SG_S$ is bifree as an $(S,S)$-biset as both left and right multiplication in a group are invertible operations.
\item $_SG_S$ is symmetric via the inversion map $g\mapsto g^{-1}$.
\item An element $g\in G$ has point-stabilizer $(c_g,S_g)$, and $c_g\in\cF(S_g,S)$ by definition.
\item If $\varphi\in\cF(P,S)$ is induced by $g$, left multiplication by $g$ yields ${_PG_S}\cong{_P^\varphi G_S}$ and right multiplication by $g^{-1}$ yields ${_SG_P}\cong{_SG_P^\varphi}$, so $_SG_S$ is $\cF$-invariant. 
\item $|G|/|S|$ is prime to $p$ as $S$ is a Sylow $p$-subgroup of $G$.
\end{enumerate}

The proof that an $S$-invariant $k$-basis $X$ of $B$ is $\cF$-semicharacteristic will similarly amount to a checklist verification, except that Condition \ref{cond:v} will not  hold in general.   Indeed, in \S4 we will see that an $S$-invariant $k$-basis of $B$ is $\cF$-characteristic if and only if $B$ has maximal defect, i.e., $S$ is a defect group of $B$.\\

\noindent{\bf \S3.  The proof of Theorem \ref{thm:main}.}

\noindent Only Conditions \ref{cond:ii} and \ref{cond:iv} in the definition of $\cF$-semicharacteristic biset are not obvious for $X$.  We prove these separately in the following two propositions.

%\emph{Aside.}  It should be noted that this proof is somewhat redundant:  The symmetry of $X$ can be shown to follow from $\cF$-stability and basic properties of fusion systems, most crucially that if $\varphi\in\cF(P,S)$ then $\varphi^{-1}\in\cF(\varphi(P),S)$. Moreover, the basic idea behind the proofs of Propositions \ref{prop:F_stability_of_block_basis} and \ref{prop:invariant_bases_of_symmetric_algebra_are_symmetric} is the same:  Transform the $S$-invariant $k$-basis $X$ into another $S$-invariant $k$-basis $X'$ in a way that twists the biset structure appropriately, then use Proposition \ref{prop:key_facts}\ref{proppart:invariant_bases_well_defined} to show that $X\cong X'$ as $(S,S)$-bisets.  The specifics of the twisting part of the argument are different, however, and it seems possible that Proposition \ref{prop:invariant_bases_of_symmetric_algebra_are_symmetric} may be of interest in its own right, so we have elected to include both. \emph{End aside.}

\begin{prop}\label{prop:F_stability_of_block_basis}
An $S$-invariant $k$-basis $X$ of the block algebra $B$ is $\cF$-invariant.
\begin{proof}
Let $P\leq S$ and $\varphi\in\cF(P,S)$ be given.  We  show ${_P^\varphi X_S}\cong{_PX_S}$ as $(P,S)$-bisets; the proof that ${_SX_P^\varphi}\cong{_SX_P}$ as $(S,P)$-bisets follows the same argument.

Fix $g\in G$ inducing $\varphi\in\cF(P,S)$, so $gug^{-1}=\varphi(u)$ for all $u\in P$.  Set $X'=g^{-1}\cdot X$.  Since $B$ is an interior $G$-algebra, we have $g^{-1}\cdot X=\overline g^{-1}X$, so $X'$ is a $k$-basis of $B$ multiplied by a unit.  In particular, $X'$ is also a $k$-basis of $B$.

For any $g^{-1}\cdot x\in X'$ and $u\in P$, we have $u\cdot(g^{-1}\cdot x)=g^{-1}\cdot(\varphi(u)\cdot x)$.  As $\varphi(u)\in S$ and $X$ is $S$-invariant, we conclude that $u\cdot X'=X'$, and thus $X'$ is a $P\times S$-invariant $k$-basis of $B$.  (That $X'$ is invariant under the right $S$-action is obvious.)  Proposition \ref{prop:key_facts}\ref{proppart:invariant_bases_well_defined}  implies that $X\cong X'$ as $(P,S)$-bisets, say via some bijection $f:X\to X'$.  Then the composite bijection $F:X\to X'\to X:x\mapsto f(x)\mapsto g\cdot f(x)$ satisfies, for all $u\in P$ and $x\in X$,
\[
F(u\cdot x)=g\cdot f(u\cdot x)=gu\cdot f(x)=\varphi(u)\cdot(g\cdot f(x))=\varphi(u)\cdot F(x).
\]
Again it is obvious that $F(x\cdot s)=F(x)\cdot s$ for all $s\in S$.  Thus $F$ is an isomorphism of $(P,S)$-bisets ${_PX_S}\cong{_P^\varphi X_S}$, so $X$ is $\cF$-invariant.
\end{proof}
\end{prop}

In order to prove the symmetry of  $X$, we first make a small detour.

Let $A$ be a finite dimensional $k$-algebra.  A \emph{symmetrizing form} is a $k$-linear map ${\lambda:A\to k}$ whose kernel contains no nontrivial left (or right) ideals and such that $\lambda(a_1a_2)=\lambda(a_2a_1)$ for all $a_1,a_2\in A$.  If $A$ possesses a symmetrizing form, $A$ is a \emph{symmetric $k$-algebra}.  Equivalently, $A$ is a  symmetric $k$-algebra if the regular $(A,A)$-bimodule $_AA_A$ is isomorphic to its linear dual $A^\ast:=\Hom_k(A,k)$ as $(A,A)$-bimodules.  (See, e.g., \cite[\S1.6]{ThevenazBook} for a review of this standard material, and a more general version of Lemma \ref{ref:corner_algebras_of_symmetric_algebras_are_symmetric} below.)

If $H$ is a finite group, the $k$-algebra $A$ is an \emph{interior $p$-permutation $H$-algebra} if $A$ is an interior $H$-algebra and  for any $p$-subgroup $P\leq H$, $A$ possesses a $k$-basis $Y=Y_P$ that is invariant under the left and right $P$-action.

\begin{prop}\label{prop:invariant_bases_of_symmetric_algebra_are_symmetric}
Let $A$ be a symmetric interior $p$-permutation  $H$-algebra.  If $P\leq H$ is a $p$-subgroup and $Y$ is a $P$-invariant $k$-basis of $A$, then $Y$ is symmetric as a $(P,P)$-biset.
\begin{proof}
For ease of expression, enumerate  the elements of the $k$-basis: $Y=\{y_1,y_2,\ldots,y_n\}$.  Let $Y^*=\{y_1^*,y_2^*,\ldots,y_n^*\}$ be the dual basis of $A^*$ with respect to $Y$, i.e., $y_i^*(y_j)=\delta_{ij}$.

Let $\lambda$ be the symmetrizing form of $A$.  For each any $a\in A$, let $\lambda_a:A\to k$ be the linear functional $\lambda_a:a'\mapsto\lambda(aa')$.  Clearly the assignment $a\mapsto \lambda_a$ defines a $k$-linear map $A\to A^*$.  Moreover, if $\lambda_a$ is the trivial functional, then $\lambda(Aa)=\lambda(aA)=0$, or the left ideal $Aa$ is contained in the kernel of $\lambda$.  The assumption that $\ker\lambda$ contains no nontrivial left ideals forces $a=0$, so $\lambda_-:A\to A^*$ is a $k$-injection.  As $A$ is finite dimensional over $k$, we conclude that $\lambda_-$ is a $k$-isomorphism.  

Thus, for each $y_i\in Y$, there is a unique $\check y_i\in A$ such that $\lambda_{\check y_i}=y_i^*$.  In other words, $\check y_i$ is defined by $\lambda(\check y_iy_j)=\delta_{ij}$.  Let $\check Y=\{\check y_1,\check y_2,\ldots,\check y_n\}$.  Clearly $\check Y$ is a $k$-basis for $A$.

Consider now, for $u_1,u_2\in P$, $\check y_i\in\check Y$, and $y_j\in Y$, we have
\[
\lambda((u_1\cdot\check y_i\cdot u_2)y_j)=\lambda(\overline u_1\check y_i\overline u_2y_j)=\lambda(\check y_i\overline u_2 y_j\overline u_1)=\lambda(\check y_i(u_2\cdot y_j\cdot u_1)).
\]
As $Y$ is $P$-invariant, $u_2\cdot y_j\cdot u_1\in Y$, so the above common value is $1$ if $u_2\cdot y_j\cdot u_1=y_i$, or equivalently $y_j=u_2^{-1}\cdot y_i\cdot u_1^{-1}$, and $0$ otherwise. This implies that
\[
u_1\cdot \check y_i\cdot u_2=(u_2^{-1}\cdot y_i\cdot u_1^{-1})^\vee\in\check Y.
\]
This shows both that $\check Y$ is $P$-invariant and that $\check Y\cong Y^\circ$ as $(P,P)$-bisets.

As $\check Y$ is a $P$-invariant $k$-basis of $A$, Proposition \ref{prop:key_facts}\ref{proppart:invariant_bases_well_defined} implies $Y\cong \check Y$ as $(P,P)$-bisets as well.  Combining these isomorphisms gives $Y\cong Y^\circ$, so $Y$ is a symmetric $(P,P)$-biset.
\end{proof}
\end{prop}

The last well-known ingredient is that algebra direct summands of symmetric algebras are symmetric:

\begin{lemma}\label{ref:corner_algebras_of_symmetric_algebras_are_symmetric}
Let $A$ be a symmetric $k$-algebra with symmetrizing form $\lambda$.  If $e\in A$ is idempotent, then the corner algebra $eAe$ is symmetric with symmetrizing form $\lambda|_{eAe}$.
\begin{proof}
Let $\overline\lambda=\lambda|_{eAe}$.  Clearly $\overline\lambda (xy)=\overline\lambda(yx)$ for all $x,y\in eAe$, so it suffices to show that $\overline\lambda$ contains no nonzero left ideals of $eAe$.

Let $J\subseteq eAe$ be a left ideal of $eAe$ contained in $\ker\lambda$.  As $e$ is the identity element of $eAe$, we have $J=eJe$.  Consider the left idea $AJ$ of $A$ generated by $J$.  Then we have
\[
\lambda(AJ)=\lambda(AeJe)=\lambda(eAeJ)=\lambda(J)=\overline\lambda(J)=0,
\]
so that $AJ$ is a left $A$-ideal contained in $\ker\lambda$.  Thus $J\subseteq AJ=0$, and we have verified that $\overline\lambda$ is a symmetrizing form for $eAe$.
\end{proof}
\end{lemma}

\begin{proof}[Proof of Theorem \ref{thm:main}]
Let $X$ be an $S$-invariant $k$-basis for the block $B$, whose existence is guaranteed by Proposition \ref{prop:block_invariant_basis}.  Proposition \ref{prop:key_facts}\ref{proppart:subsets} implies that $X$ is isomorphic to a $(S,S)$-subbiset of $_SG_S$, which we have already observed to be $\cF$-characteristic.  Bifreeness and $\cF$-generation are clearly properties inherited by sub-$(S,S)$-bisets, so we have verified Conditions \ref{cond:i} and \ref{cond:iii} in the definition of $\cF$-semicharacteristic bisets.  

The group algebra $kG$ is easily seen to be symmetric with symmetrizing form $\lambda:\sum\limits_{g\in G}\alpha_gg\mapsto\alpha_1$.  The block algebra $B=kGb=b(kG)b$ is the group algebra cut by an idempotent, so $B$ is symmetric by Lemma \ref{ref:corner_algebras_of_symmetric_algebras_are_symmetric}.  Therefore $B$ is a symmetric interior $p$-permutation $S$-algebra, so $X$ is symmetric by Proposition \ref{prop:invariant_bases_of_symmetric_algebra_are_symmetric}, and we have satisfied Condition \ref{cond:ii}.

Finally, $X$ is $\cF$-invariant by Proposition \ref{prop:F_stability_of_block_basis}, which verifies Condition \ref{cond:iv}.  This completes the proof that $X$ is $\cF$-semicharacteristic.
\end{proof}

\noindent{\bf \S4.  Some implications.}

\noindent  In \cite{GelvinReehMinimalCharacteristicBisets} it is shown that the monoid of $\cF$-characteristic bisets possesses a natural basis.   In particular, the $S$-invariant $k$-basis $X$ of $B$ decomposes uniquely in terms of this basis, which significantly constrains the $(S,S)$-biset structure of $X$. We recall the characterization of this basis now:

 If $(\psi,Q)$ is a twisted diagonal subgroup of $S\times S$, let $[\psi,Q]$ denote the transitive $(S,S)$-biset that contains an element whose point stabilizer is $(\psi,Q)$.  
 %This $(S,S)$-orbit can be realized as ${S\times S/\sim}$ where $(s_1\varphi(q),s_2)\sim(s_1,qs_2)$ for all $q\in Q$, $s_1,s_2\in S$.  
 If $\Omega$ is a $\cF$-semicharacteristic biset, Conditions \ref{cond:i} and \ref{cond:iii} imply that
 \[
 \Omega\cong\coprod_{\substack{Q\leq S\\\psi\in\cF(Q,S)}}c_{(\psi,Q)}(\Omega)\cdot[\psi,Q]
 \]
for $c_{(\psi,P)}(\Omega)\in\NN$. 

Let $P\leq S$ be \emph{fully $\cF$-normalized}: The order of $N_S(P)$ is maximal among the orders of the $S$-normalizers of $Q$ when $Q={^gP}\leq S$ for some $g\in G$, or equivalently $N_S(P)\in\Syl_p(N_G(P))$.  Then by \cite[Theorems 4.5 and 5.3]{GelvinReehMinimalCharacteristicBisets} there is a unique $\cF$-semicharacteristic biset $\Omega_P=\Omega_P^\cF$ such that $c_{(\id,P)}(\Omega_P)=1$ and if $Q$ is any fully $\cF$-normalized subgroup with $c_{(\id,Q)}(\Omega_P)\neq 0$ then $Q\cong_\cF P$.  Moreover, if $c_{(\psi,Q)}(\Omega_P)\neq0$ then $(\psi,Q)$ is $\cF\times\cF$-subconjugate to $(\id,P)$, i.e., there exist $\chi\in\cF(Q,P)$ and $\chi'\in\cF(\psi(Q),P)$ such that $\chi=\chi'\circ\psi$.

These $\{\Omega_P\}$, as $P$ ranges over a chosen set $[Cl(\cF)]_{fn}$ of fully $\cF$-normalized representatives of the $\cF$-conjugacy classes of subgroups of $S$, form a basis for the monoid of $\cF$-semicharacteristic bisets.  Thus our arbitrary $\cF$-semicharacteristic biset $\Omega$ can be uniquely written
\[
\Omega\cong\coprod_{P\in[Cl(\cF)]_{fn}}c_P(\Omega)\cdot\Omega_P.
\]
This applies in particular to the case that $\Omega=X$ is the $S$-invariant $k$-basis of $B$.  

Even more information can be obtained through consideration of the Brauer map.  We recall basic well-known facts from the literature without proof; see, e.g., \cite{ThevenazBook} for a full treatment.  

If $A$ is an interior $G$-algebra and $H\leq G$, let $A^H:=\left\{a\in A\ \big|\ h\cdot a\cdot h^{-1}=a\ \forall\ h\in H\right\}$ denote the \emph{$H$-fixed subalgebra}.  If $K\leq H$, let $\textrm{tr}_K^H:A^K\to A^H$ denote the trace map $a\mapsto\sum\limits_{h\in[H/K]}h\cdot a\cdot h^{-1}$, where $[H/K]$ is a chosen set of coset representatives of $H/K$.  The subalgebra $A_\lneq^H:=\sum\limits_{K\lneq H}\textrm{tr}_K^H(A^K)$ is an ideal of $A^H$.  The \emph{Brauer quotient of $A$ at $H$} is $A(H):=A^H/A_\lneq^H$ and the \emph{Brauer map of $A$ at $H$} is the natural quotient $\textrm{br}_H:A^H\to A(H)$.  $A(H)=0$ unless $H$ is a $p$-subgroup of $G$.  

In the special case that $A=kG$ and $P\leq G$ is a $p$-subgroup, $A(P)\cong kC_G(P)$.  This reflects the more general fact that if $A$ is an interior $p$-permutation $G$-algebra, $P\leq G$ is a $p$-subgroup, and $Y$ a $P$-invariant $k$-basis of $A$, then the image of $Y^P:=Y\cap A^P$ under the Brauer map is a $k$-basis for $A(P)$.  In particular, $Y^P\neq\emptyset$ if and only if $A(P)\neq 0$.

If $A^G$ is local (for example, our block algebra $B$), a \emph{defect group} of $A$ is a maximal $p$-subgroup $D\leq G$ such that $A(D)\neq 0$. $D$ is well-defined up to $G$-conjugacy, and if $P\leq G$ is any $p$-subgroup such that $A(P)\neq 0$, then $P$ is $G$-subconjugate to $D$.

Putting all this together, we obtain:

\begin{cor}\label{cor:defect_basis}
Let $X$ be an $S$-invariant $k$-basis of the block algebra $B$ whose defect group $D$ is chosen to lie in $[Cl(\cF)]_{fn}$. Then $X$ contains a copy of $\Omega_P$ only if $P\leq_\cF D$.  Moreover, the number of copies of $\Omega_D$ contained in $X$ is prime to $p$.
\begin{proof}
Suppose that $X$ contains a copy of $\Omega_P$, which in turn contains the $(S,S)$-orbit $[\id,P]$.  As $[\id,P]^P\neq\emptyset$, we have $B(P)\neq0$.  Our characterization of defect groups then implies the first claim.

Observe now that $|S|^2/|P|$ divides $|\Omega_P|$:  The $(S,S)$-orbit $[\psi,Q]$ has order exactly $|S|^2/|Q|$, and the only such orbits that appear in $\Omega_P$ must satisfy that $(\psi,Q)$ is $\cF\times\cF$-subconjugate to $(\id,P)$, so that in particular $|Q|\leq|P|$. By \cite[Theorem 1]{BrauerNotesOnRepresentations}, the greatest power of $p$ that divides  $\dim_k(B)$ is $|S|^2/|D|$.  As all $P\in[Cl(\cF)]_{fn}$ such that $X$ contains a copy of $\Omega_P$ are $\cF$-subconjugate to $D$, and for all such $P$ of order strictly less  than $|D|$ the size of $\Omega_P$ has $p$-part strictly greater than $|S|^2/|D|$, we conclude that there must be a $p'$-number of copies of $\Omega_D$ in $X$. 
\end{proof}
\end{cor}

In particular, Corollary \ref{cor:defect_basis} implies that $B$ has an $\cF$-characteristic $S$-invariant $k$-basis if and only $B$ is of maximal defect, as claimed at the end of \S2.

Corollary \ref{cor:defect_basis} can also be seen to imply some well-established facts in the literature:
\begin{itemize}
\item In \cite[Theorem 3]{GreenRemarksDefectGroups} it is proved that a defect group $D$ of $B$ is a Sylow intersection subgroup, $D=S\cap{S^g}$, and that moreover $g$ can be chosen to lie in $C_G(D)$.  
\begin{itemize}
\item An $S$-invariant $k$-basis of $B$ contains $\Omega_D$, which contains $[\id,D]$, so by Proposition \ref{prop:key_facts}\ref{proppart:subsets} $_SG_S$ must contain an element $g$ with stabilizer $(\id,D)$.  As already noted, $\Stab(g)=(c_g,S\cap{S^g})$, from which the result follows.
\end{itemize}
\item Alperin and Green (each crediting the other, cf. \cite[\S 6]{AlperinFusion} and \cite[Theorem 4]{GreenRemarksDefectGroups}) show that, in our terminology, $D$ can be chosen to be fully $\cF$-normalized in $S$.
\begin{itemize}
\item $D$ is $G$-conjugate to $P\in[Cl(\cF)]_{fn}$, which is by definition fully $\cF$-normalized.
\end{itemize}
\end{itemize}

The sketched proofs we offer for these basic facts are morally the same as those found in \cite{GreenRemarksDefectGroups}, so we will not elaborate further.  These points are raised mainly because we find it interesting that the proofs are essentially contained in the characterization of the $\cF$-semicharacteristic biset basis $\{\Omega_P\}$.

We conclude by using Theorem \ref{thm:main} to give an new perspective on some well-known results in block theory (see, e.g., \cite[Corollary 3.11]{FeitRepresentationTheory}).  We say that $G$ is of \emph{characteristic $p$} if $C_G(O_p(G))\leq O_p(G)$, and is of \emph{local characteristic $p$} if $N_G(P)$ is of characteristic $p$ for all nonidentity $p$-subgroups $P\leq G$.

\begin{cor}
If $G$ is of characteristic $p$, then $kG$ has a unique block.  If $G$ is of local characteristic $p$, then all nonprincipal blocks of $kG$ are of defect $0$. 
\begin{proof}
The first claim follows from \cite[Theorem 6.7]{GelvinReehMinimalCharacteristicBisets}, which implies that if $G$ is of characteristic $p$ then ${_SG_S}\cong\Omega_S$ as $(S,S)$-bisets.  In this case, $_SG_S$  cannot be broken into smaller $\cF$-semicharacteristic bisets, so the principal block $B_0$ must have an $S$-invariant $k$-basis isomorphic to $_SG_S$.  It follows that $B_0=kG$, and the claim is proved.

More generally, let $B$ be a block of $kG$ with defect group $D\in[CL(\cF)]_{fn}$ (which may be trivial).  By Corollary \ref{cor:defect_basis}, the $S$-invariant $k$-basis $X$ of $B$ must contain a copy of $\Omega_D$, which in turn implies that there is some $g\in G$ with point-stabilizer $(\id,D)$.  This element must satisfy $D=S\cap{S^g}$ and $g\in C_G(D)\leq N_G(D)$.

As $D\leq S$ was chosen to be fully $\cF$-normalized, we have $N_S(D)\in\Syl_p(N_G(D))$.  Thus ${O_p(N_G(D))\leq{N_S(D)\cap{N_S(D)^g}}\leq S\cap{S^g}}=D$.  As $D$ is clearly a normal $p$-subgroup of $N_G(D)$, it follows that $O_p(N_G(D))=D$.  Thus $g\in C_G(O_p(N_G(D))$.

If we assume now that $G$ is of local characteristic $p$, the above implies that must have either $D=1$ or $g\in O_p(N_G(D))= D$.  In the first case $B$ is a block of defect $0$.  Thus we may assume $g\in D=S\cap{S^g}$, so that $D=S$.  By Corollary \ref{cor:defect_basis} again, we see that every block of positive defect has an $S$-invariant $k$-basis that contains a copy of $\Omega_S$, and so in particular an $(S,S)$-orbit isomorphic to $[\id,S]$.  The same argument as above shows that $_SG_S$ has a unique such orbit, which is already accounted for in the $S$-invariant $k$-basis of the principal block $B_0$.  Thus we see that if $G$ is of local characteristic $p$, any block of $kG$ with positive defect must be the principal block, and the result is proved.
\end{proof}
\end{cor}

\bibliography{Sources}
\bibliographystyle{alphanum}

\end{document}